\documentclass[12pt]{amsart}
\usepackage{amssymb}
\usepackage{aurical}
\usepackage{graphicx,psfrag}

\newtheorem{thm}{Theorem}[section]
\newtheorem{lem}[thm]{Lemma}
\newtheorem{cor}[thm]{Corollary}
\newtheorem{prop}[thm]{Proposition}
\newtheorem{ex}[thm]{Example}

\newtheorem*{prob*}{Open problem}

\theoremstyle{definition}

\newtheorem{defi}[thm]{Definition}

\theoremstyle{remark}

\newtheorem{rem}[thm]{Remark}
\newtheorem*{rem*}{Remark}

\DeclareMathOperator{\id}{id}
\DeclareMathOperator{\s}{span}

\newcommand{\kringel}{\mathbin{\raise1pt\hbox{$\scriptstyle\circ$}}}
\newcommand{\pkt}{\mathbin{\raise0pt\hbox{$\scriptstyle\bullet$}}}

\newcommand{\C}{\mathbb{C}}

\newcommand{\Z}{\mathbb{Z}}

\newcommand{\ad}{\mathop{\rm ad}}

\newcommand{\End}{\mathop{\rm End}}
\newcommand{\Der}{\mathop{\rm Der}}

\newcommand{\diag}{\mathop{\rm diag}}

\newcommand{\Lf}{\mathfrak{f}}
\newcommand{\Lg}{\mathfrak{g}}
\newcommand{\Lh}{\mathfrak{h}}

\newcommand{\Ll}{\mathfrak{l}}

\newcommand{\Ls}{\mathfrak{s}}

\newcommand{\CX}{\mathcal{X}}

\newcommand{\abs}[1]{\lvert#1\rvert}

\newcommand{\spa}{{\rm span}}

\newcommand{\al}{\alpha}
\newcommand{\be}{\beta}
\newcommand{\ga}{\gamma}
\newcommand{\de}{\delta}

\newcommand{\la}{\lambda}
\newcommand{\om}{\omega}

\newcommand{\ra}{\rightarrow}

\renewcommand{\phi}{\varphi}

\begin{document}

\title[Periodic derivations]{Periodic derivations and prederivations of Lie algebras}

\author[D. Burde]{Dietrich Burde}
\author[W. Moens]{Wolfgang Alexander Moens}
\address{Fakult\"at f\"ur Mathematik\\
Universit\"at Wien\\
  Nordbergstrasse 15\\
  1090 Wien \\
  Austria}
\email{dietrich.burde@univie.ac.at}
\address{Fakult\"at f\"ur Mathematik\\
Universit\"at Wien\\
  Nordbergstrasse 15\\
  1090 Wien \\
  Austria}
\email{wolfgang.moens@univie.ac.at}
\date{\today}

\subjclass[2000]{Primary 17B10, 17B25}
\thanks{The authors were supported by the FWF, Projekt P21683.}

\begin{abstract}
We consider finite-dimensional complex Lie algebras admitting a periodic
derivation, i.e., a nonsingular derivation which has finite multiplicative order.
We show that such Lie algebras are at most two-step nilpotent and give
several characterizations, such as the existence of gradings by sixth roots of unity, or
the existence of a nonsingular derivation whose inverse is again a derivation. 
We also obtain results on the existence of periodic prederivations. In this context
we study a generalization of Engel-$4$-Lie algebras.
\end{abstract}

\maketitle

\section{Introduction}

Let $\Lg$ be a Lie algebra over a field $k$. A derivation of $\Lg$ is called
{\it nonsingular}, if it is injective as a linear map. Lie algebras admitting 
nonsingular derivations have been studied in many different contexts.
First, they play an important role in the existence question of left-invariant 
affine structures on Lie groups. Here nonsingular derivations arise as a special case 
of invertible $1$-cocylces for the Lie algebra cohomology with coefficients in $\Lg$-modules $M$ 
with $\dim(M)=\dim (\Lg)$. If $D$ is a nonsingular derivation, then the formula
$x\cdot y=D^{-1}([x,D(y)])$ defines a left-symmetric structure on $\Lg$.
For a survey on this topic see \cite{BU24} and the references therein. 
An important structure result for Lie algebras $\Lg$ of characteristic zero with a nonsingular 
derivation has been given by Jacobson \cite{JAC}. It says that such Lie algebras must be nilpotent. 
For Lie algebras in prime characteristic the situation is more complicated.
In that case there exist non-nilpotent Lie algebras, even simple ones, which admit
nonsingular derivations \cite{BKK}. This is very interesting for the coclass theory of pro-$p$ groups
and Lie algebras, see \cite{SHA} and the references given therein. In this context also the
orders of nonsingular derivations have been studied. This leads naturally to a subclass 
of nonsingular derivations, given by {\it periodic} derivations, where
the derivation has finite multiplicative order. Again it is interesting to obtain structure
results on Lie algebras admitting periodic derivations.
Whereas this has been studied intensively in prime characteristic, there seems to be
only one result for the characteristic zero case, which is proved \cite{KK}. 
It says that such Lie algebras are abelian, if the order of the periodic derivation is {\it not} 
a multiple of six. There is nothing said in \cite{KK} on Lie algebras admitting a periodic derivation of 
order which {\it is} a multiple of six. The aim of this paper is to close this gap by characterizing 
such Lie algebras. We first prove that such Lie algebras are abelian or two-step nilpotent. 
Then we show that the existence of a periodic derivation is equivalent to the existence
of a so called {\it hexagonal grading} by sixth roots of unity. This again is equivalent to the
existence of a nonsingular derivation whose inverse is again a derivation.
As it turns out, not all two-step nilpotent Lie algebras admit a periodic derivation. 
This leads us to consider further invariants of two-step nilpotent Lie algebras admitting
periodic derivations. \\
In the last section, we study the much more difficult case of periodic prederivations.
A Lie algebra admitting a periodic prederivation is again nilpotent. However, it is not clear
how to find a good bound on the nilpotency class. In many cases we can prove that such Lie algebras
are at most $4$-step nilpotent. This involves Engel-$4$-Lie algebras and 
so called pre-Engel-$4$ Lie algebras. On the other hand, however, we also construct a $5$-step
nilpotent Lie algebra admitting a periodic prederivation.


\section{Periodic derivations}

Let $\Lg$ be a Lie algebra. We always assume that $\Lg$ is finite-dimensional and complex, 
if not mentioned otherwise. Denote by $\Der(\Lg)$ the Lie algebra of derivations of $\Lg$.
Periodic derivations have been defined in \cite{KK}.

\begin{defi}
Let $\Lg$ be a Lie algebra. A derivation $D$ of $\Lg$ is called {\it periodic},
if there is an integer $m\ge 1$ such that $D^m=\id$.
\end{defi}

If $\Lg$ has a periodic derivation $D$ with $m=1$, then we have $[x,y]=D([x,y])=
[D(x),y]+[x,D(y)]=2[x,y]$, so that $\Lg$ is abelian. 
Conversely, an abelian Lie algebra has periodic derivations of any possible order $m\ge 1$.
Indeed, just define $D=\zeta_m\id \in \End(\Lg)=\Der(\Lg)$, where $\zeta_m$ is a 
primitive $m$-th root of unity. \\[0.2cm]
Our aim is to characterize complex Lie algebras admitting a periodic derivation. 
We need an elementary lemma related to roots of unity.

\begin{lem}\label{omega}
Let $\al,\be$ be complex numbers with $\abs{\al}=\abs{\be}=\abs{\al+\be}=1$.
Then $\be=\om \al$ with a primitive third root of unity $\om$.
\end{lem}

\begin{proof}
The points $0$, $\al$ and $\al+\be$ are the vertices of an equilateral triangle.
Hence $\be=\om \al$. To see this algebraically, let $\ga=-(\al+\be)$. Then $\al+\be+\ga=0$ and 
$\abs{\al}=\abs{\be}=\abs{\ga}=1$. We may write $\be=e^{i\phi}\al$ and $\ga=e^{i\psi}\al$.
Substituting in  $\al+\be+\ga=0$ one obtains $1+e^{i\phi}+e^{i\psi}=0$. Equating real and imaginary 
parts yields
$\psi=-\phi$ and $\cos (\phi)=-1/2$. We may take $\phi$ to be positive, 
so that $\phi=2\pi/3$ and $\psi=-2\pi/3$. Then $\om=e^{i\phi}=e^{2\pi i/3}$ 
and  $\be=\om \al$.
\end{proof}

\begin{cor}\label{roots1}
Let $\al,\be,\ga \in \C$. Then at least one of the numbers $\al,\be,\ga,\al+\ga,\be+\ga,
\al+\be+\ga$ is not an $m$-th root of unity.
\end{cor}

\begin{proof}
Assume that all of these numbers are $m$-th roots of unity. Then by lemma $\ref{omega}$ 
we have $\al=\om^r(\be+\ga)=\om^t$ and $\be=\om^s\ga$ for some $r,s,t =\pm 1$. 
This implies $\om^r(1+\om^s)=\om^t$. Raising this to the third power yields
$-1=(1+\om^s)^3=1$, a contradiction.
\end{proof}

\begin{cor}\label{roots2}
Let $\al_1, \ldots , \al_4$ be $m$-th roots of unity. Then at least one of the numbers 
$\al_i+\al_j$ for $i\neq j$ is not an $m$-th root of unity.
\end{cor}

\begin{proof}
Assume that all of these numbers are $m$-th roots of unity. Then by lemma $\ref{omega}$ 
we have $\al_2=\om^r \al_1$,  $\al_3=\om^s \al_1$ and $\al_4=\om^t \al_1$ for
some $r,s,t \in \{1,-1\}$. Now two of these exponents must coincide, say $r=s$. 
This yields $1=\abs{\om^r+\om^s}=\abs{2\om^r}=2$, a contradiction.
\end{proof}

The only result on periodic derivations in characteristic zero we could find is the
following proposition of \cite{KK}.

\begin{prop}\label{multiple}
Let $\Lg$ be a nonabelian Lie algebra of characteristic zero, which admits a periodic 
derivation of order $m$. Then $m$ is a multiple of six.
\end{prop}

\begin{proof}
There is a proof given in \cite{KK} using a binomial formula and determinants, which works for
arbitrary characteristic. For the complex numbers there is a shorter
argument available. Let $D$ be a periodic derivation of order $m$. Since any complex
endomorphism of finite order is semisimple, $D$ is semisimple. Since $\Lg$ is nonabelian 
there exist eigenvectors $u$ and $v$ with eigenvalues
$\al$ and $\be$ such that $[u,v]$ is a nonzero eigenvector with eigenvalue $\al+\be$.
Indeed, $D([u,v])=[D(u),v]+[u,D(v)]=(\al+\be)[u,v]$.
This means $\al^m=\be^m=(\al+\be)^m=1$, so that $\be=\al \om$ with a primitive third root of 
unity $\om$, by lemma $\ref{omega}$. Then we have $\al+\be=\al(1+\om)$. Raising this to the
$m$-th power we obtain $(1+\om)^m=1$, with $1+\om$ being a primitive sixth root of unity, 
so that $6\mid m$.
\end{proof}

The result can be reformulated as a structure result for $\Lg$. A Lie algebra admitting
a periodic derivation of order $m$ coprime to six is abelian. Jacobson has proved \cite{JAC},
that a Lie algebra of characteristic zero admitting a nonsingular derivation must be
nilpotent. Since a periodic derivation is nonsingular as a linear transformation, we 
obtain the following structure result.

\begin{prop}
Let $\Lg$ be a Lie algebra admitting a periodic derivation. Then $\Lg$ is nilpotent.
\end{prop}

In characteristic $p>0$ this need not be true. There are non-nilpotent Lie algebras, even simple 
Lie algebras, which have periodic derivations, see \cite{SHA}. 
For complex Lie algebras we obtain a stronger result than the above proposition.

\begin{prop}\label{2-step}
Let $\Lg$ be a Lie algebra admitting a periodic derivation. Then
$\Lg$ is at most two-step nilpotent.
\end{prop}

\begin{proof}
Let $D$ be a periodic derivation of $\Lg$. Hence it is semisimple.
Assume that $[\Lg,[\Lg,\Lg]]\neq 0$. Then there exist eigenvectors $x,y,z$ of $D$ 
with eigenvalues $\al,\be,\ga$ such that $[x,[y,z]]\neq 0$. Hence $[y,z]\neq 0$. By
the Jacobi identity we may assume that also $[y,[x,z]]\neq 0$, hence $[x,z]\neq 0$. 
We conclude that $\al,\be,\ga,\al+\ga,\be+\ga,\al+\be+\ga$ are all $m$-th roots 
of unity. This contradicts corollary $\ref{roots1}$.
\end{proof}

\begin{prop}\label{inverse}
Let $\Lg$ be a Lie algebra admitting a periodic derivation $D$. Then
the inverse $D^{-1}$ is again a derivation of $\Lg$.
\end{prop}

\begin{proof}
We know that $D$ is semisimple. Let $e_1,\ldots ,e_n$ be a basis of eigenvectors
with eigenvalues $\al_1,\ldots ,\al_n$. As in the proof of proposition $\ref{multiple}$,
we see that for two noncommuting eigenvectors $e_i$ and $e_j$ with eigenvalues
$\al_i$ and $\al_j$ we have $\al_i=\al_j\om$.
Then $D^{-1}([e_i,e_j])=[D^{-1}(e_i),e_j]+[e_i,D^{-1}(e_j)]$ follows from
\[
\al_i^{-1}+\al_j^{-1} =\al_j^{-1}(1+\om^{-1})=\al_j^{-1}(1+\om)^{-1} = (\al_i+\al_j)^{-1}.
\]
\end{proof}

If a decomposable Lie algebra $\Lg$ admits a periodic derivation of order a multiple of six, 
say $12$, then we may not obtain a periodic derivation of order six just by rescaling.
Indeed, for $\Lg=\C^2$ the linear map $D=\diag (\zeta,\zeta^2)$, with $\zeta$ being a primitive 
$12$-th root of unity, is a periodic derivation of order $12$. However, no multiple $\la D$ has 
order six: $1=(\la\zeta^2)^6=(\la\zeta)^6$ yields $1=\zeta^6=-1$, which is a contradiction.

\begin{defi}
Denote by $N(c,g)$ the free-nilpotent Lie algebra of nilpotency class $c$ and $g$ generators.
\end{defi}

Let $c=2$. The Lie algebra $N(2,1)$ is abelian, and $N(2,2)$ is the Heisenberg Lie algebra,
with $[x_1,x_2]=x_3$. It admits periodic derivations, e.g., $D=\diag (1,\om,1+\om)$ or 
\[
D=\begin{pmatrix} 0 & -1 & 0 \\ 1 & 1 & 0 \\ 0 & 0 & 1 \end{pmatrix}. 
\]
The latter derivation is integral, based on an element of order six in $SL_2(\Z)$. 
This example generalizes to all Heisenberg Lie algebras $\Lh_m$ of dimension $2m+1$, with
Lie brackets $[x_i,y_i]=z$ for $1\le i\le m$. Indeed,  $D=\diag (1,\ldots ,1,\om,\ldots ,\om,1+\om)$
is a periodic derivation of $\Lh_m$. \\
The Lie algebra $N(2,3)$, with basis $x_1,\ldots ,x_6$ and brackets
\[
[x_1,x_2]=x_4, \; [x_1,x_3]=x_5, \; [x_2,x_3]=x_6
\]
is our prototype of a Lie algebra admitting a periodic derivation.

\begin{ex}\label{st}
The Lie algebra $N(2,3)$ has a periodic derivation of order six, given by
\[
D=\diag (1,\om,\om^2,1+\om,1+\om^2,\om+\om^2).
\]
\end{ex}

It is easy to see that $D^{-1}=\diag (1,\om^2,\om,1+\om^2,1+\om,\om+\om^2)$. This is
again a derivation. 

In low dimensions every two-step nilpotent Lie algebra admits a periodic derivation.
This can be seen by explicit calculations, using a list of complex two-step nilpotent Lie
algebras. It also follows more easily from the results of section $4$.

\begin{prop}
Let $\Lg$ be a two-step nilpotent Lie algebra of dimension $n\le 6$. Then $\Lg$
admits a periodic derivation.
\end{prop}

\begin{cor}
Let $\Lg$ be a two-step nilpotent Lie algebra with $g\le 3$ generators, i.e. with
first Betti number $b_1\le 3$. Then $\Lg$ admits a periodic derivation.
\end{cor}

\begin{proof}
Such a Lie algebra is a quotient of $N(2,3)$. In the latter 
case it would be at most $6$-dimensional.
\end{proof}

In general, not  every two-step nilpotent Lie algebra admits a periodic derivation.
Already in dimension $7$ there are counter examples. Consider the Lie algebra 
$\Lg=N(2,4)/I_5$ in theorem $7.15$ of \cite{GAU}, with Lie brackets
\[
[x_1,x_2]=x_5, \; [x_1,x_3]=x_6, \; [x_2,x_3]=x_7,\;  [x_3,x_4]=-x_5. 
\]

\begin{ex}\label{2.13}
The two-step nilpotent Lie algebra  $\Lg=N(2,4)/I_5$ of dimension $7$ has no periodic
derivation.
\end{ex}

Indeed, it is easy to see that a derivation of $\Lg$ has eigenvalues of the form
\[
\al,\be,\ga,\be-\al,\al+\ga,\be-\ga,\al+\be-\ga.
\]
These cannot be all $m$-th roots of unity. \\[0.2cm]
In contrast to this, all two-step nilpotent Lie algebras admit a periodic {\it automorphism}
of any possible order $m\ge 1$. See \cite{JAC} for the discussion on periodic automorphisms. \\
More examples of higher dimension without periodic derivations are given by the following
result.

\begin{prop}
The free-nilpotent Lie algebra $N(2,g)$ admits a periodic derivation if and only if
$g\le 3$.
\end{prop}

\begin{proof}
The case $g\le 3$ has been treated above. 
Assume that $g\ge 4$ and that $\Lf=N(2,g)$ admits a periodic derivation $D$. Then $D$
is semisimple and we can find eigenvectors $x_1,\ldots ,x_g$ which span a subspace
complementary to $[\Lf,\Lf]$. The commutator is spanned by the linearly independent
eigenvectors $[x_i,x_j]$ for $1\le i<j\le g$. If the $x_i$ have eigenvalue $\la_i$, then
the $[x_i,x_j]$ have eigenvalue $\la_i+\la_j$, for $i\neq j$. We obtain $m$-rooths of
unity $\la_1,\la_2,\la_3,\la_4$ and $\la_i+\la_j$ for $i\neq j$. This contradicts corollary
$\ref{roots2}$.
\end{proof}

\section{Gradings}

We introduce the following grading, motivated by the periodic derivation
\[
D=\diag(\al,\be,\ga,\al+\be,\al+\ga,\be+\ga)=\diag (1,\om,\om^2,1+\om,1+\om^2,\om+\om^2)
\]
of example $\ref{st}$.

\begin{defi}\label{hex}
A Lie algebra $\Lg$ is called {\it hexagonally graded}, if it admits 
a vector space decomposition
\[
\Lg=\Lg_{\al}\oplus \Lg_{\be}\oplus \Lg_{\ga}\oplus \Lg_{\al+\be}\oplus \Lg_{\al+\ga}\oplus \Lg_{\be+\ga}
\]
with all indices being distinct complex numbers, satisfying
\begin{itemize}
\item[(1)] $[\Lg_i,\Lg_j]\subseteq \Lg_{i+j}$ for all $i,j\in 
\{\al,\be,\ga,\al+\be,\al+\ga,\be+\ga \}$
\item[(2)] $[\Lg_i,\Lg]=0$ for all $i\in \{\al+\be,\al+\ga,\be+\ga \}$.
\end{itemize}
\end{defi}

\begin{rem}
Note that a hexagonally graded Lie algebra $\Lg$ satisfies $[\Lg,[\Lg,\Lg]]=0$.
\end{rem}

The term hexagonally graded can be illustrated as follows. Let $\om=e^{2\pi i/3}$.
With $\al=1$, $\be=\om$ and $\ga=\om^2$ we have 
\[
(\al,\be,\ga,\al+\be,\al+\ga,\be+\ga)=(1,\om,\om^2,1+\om,1+\om^2,\om+\om^2),
\] 
which consists just of all the sixth roots of unity.

\begin{center}
{
\psfrag{1}{$1$}
\psfrag{2}{$1+\om$}
\psfrag{3}{$\om$}
\psfrag{4}{$\om+\om^2$}
\psfrag{5}{$\om^2$}
\psfrag{6}{$1+\om^2$}
\includegraphics{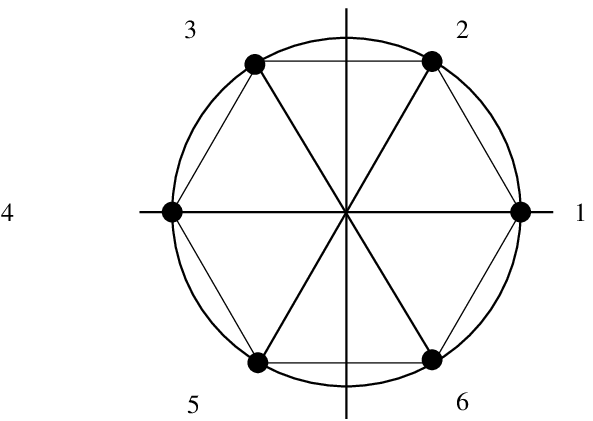}
}
\end{center}

\begin{rem}
We can rewrite the definition in a shorter way.
A Lie algebra $\Lg$ is hexagonally graded, if it admits an additive grading by sixth 
roots of unity, $\Lg=\oplus_{\zeta^6=1}\Lg_{\zeta}$, such that the subspace 
$\oplus_{\zeta^3=1}\Lg_{\zeta}$ corresponding to the third roots of unity is central.
In this formulation,  the third roots form a subgroup. Note that in the above
picture we should rotate the roots to achieve this.
\end{rem}

\begin{lem}\label{3.4}
Let $\Lg$ be hexagonally graded. Then $\Lg$ admits a periodic dervation of order six.
\end{lem}

\begin{proof}
Let $\Lg=\oplus_{\zeta^6=1}\Lg_{\zeta}$ be a hexagonal grading. We may assume that $\Lg$
is nonabelian. Define a linear map $D\colon \Lg\ra \Lg$ by its restrictions to the subspaces 
$\Lg_{\zeta}$, i.e., by $D(x)=\zeta x$ for $x\in \Lg_{\zeta}$. This is a periodic 
derivation of $\Lg$ of order six.
\end{proof}

\begin{defi}
A Lie algebra $\Lg$ is called {\it triangularly graded}, if it admits an additive grading by the 
non-zero complex numbers, $\Lg = \bigoplus_{\alpha \in \C^\times} \Lg_\alpha$ with
$[\Lg_{\al},\Lg_{\be}]\subseteq \Lg_{\al+\be}$, such that $[\Lg_\alpha,\Lg_\beta] \neq 0$ 
implies $\be=\om \al$ with a primitive third root of unity $\om$.

\end{defi}

If $\be=\om \al$, then $\al,\be,-(\al+\be)$ lie on a circle and form the vertices of an
equilateral triangle.

\begin{lem}
Let $\Lg$ be a nonabelian triangularly graded Lie algebra. Then $\Lg$ is two-step nilpotent.
\end{lem}

\begin{proof}
The proof goes exactly like the one of proposition $\ref{2-step}$.
Suppose $[\Lg,[\Lg,\Lg]] \neq 0$. Then there exist nonzero scalars $\alpha,\beta$ and $\gamma$ 
such that $[\Lg_\alpha,[\Lg_\beta,\Lg_\gamma]] \neq 0$. This implies 
$[\Lg_\beta,\Lg_\gamma] \neq 0$. By the Jacobi identity we may assume that $ [\Lg_\alpha,\Lg_\gamma] $
is nonzero. Now the conditions on $\al,\be,\ga$ contradict again corollary  $\ref{roots1}$.
\end{proof}

It is clear that a hexagonally graded Lie algebra is also triangularly graded. But also
the converse is true.

\begin{prop}\label{hextrig}
A Lie algebra $\Lg$ is hexagonally graded if and only if it is triangularly graded.
\end{prop}

\begin{proof}
Let $\Lg=\oplus \Lg_{\al}$ be triangularly graded. Then $[\Lg,\Lg]=\oplus_{\al}W_{\al}$ for
linear subspaces $W_{\al}$ of $\Lg_{\al}$. Write $\Lg_{\al}= V_{\al}\oplus W_{\al}$ with a 
complementary vector space $V_{\al}$ and define $V=\oplus_{\al}V_{\al}$. Then $\Lg=
V\oplus [\Lg,\Lg]=\bigoplus_{\al}(V_{\al}\oplus W_{\al})$.
Define an equivalence relation on $\C^{\times}$ by
\[
\al\sim \be \Leftrightarrow \al^3=\be^3.
\]
For $\al\sim \be$ we have $[V_{\al},V_{\be}]\subseteq [\Lg_{\al},\Lg_{\be}]$, which maybe 
nonzero except for $\al=\be$ where $[\Lg_{\al},\Lg_{\al}]\subseteq \Lg_{2\al}=0$.
For $\al\not\sim \be$ we have $[V_{\al},V_{\be}]\subseteq [\Lg_{\al},\Lg_{\be}]=0$.  

Now consider for each $\al\in \C^{\times}$, with $\be=\om \al$, $\ga=\om^2\al$ 
the linear subspaces
\[
\Lg(\al)=V_{\al}\oplus V_{\be}\oplus V_{\ga} \oplus [V_{\al},V_{\be}]\oplus [V_{\al},V_{\ga}]\oplus
[V_{\be},V_{\ga}]  
\]
of $\Lg$. We have $\Lg(\al)=\Lg(\al')$ for $\al \sim \al'$ and $[\Lg(\al),\Lg(\al')]=0$ for
$\al \not\sim \al'$, because $\Lg$ is two-step nilpotent and $[V_{\al},V_{\al'}]=0$ in this case.
We want to show that $\Lg$ is a direct sum of ideals $\Lg(\al)$, i.e.,
\[
\Lg=\bigoplus_{\al/\sim}\Lg(\al).
\]
This would immediately imply that $\Lg$ itself is hexagonally graded.
First it is easy to see that all $\Lg(\al)$ are Lie subalgebras. 
Then consider the sum $\sum_{\al/\sim} \Lg(\al)$ of commuting subalgebras.
It is a Lie subalgebra of $\Lg$ and contains $V$. Hence it coincides with $\Lg$, since
$V$ generates $\Lg$ as a Lie algebra.
Then we have $[\Lg,\Lg(\al)]\subseteq [\Lg(\al),\Lg(\al)]\subseteq \Lg(\al)$, so that 
$\Lg(\al)$ is a Lie ideal of $\Lg$. 
It remains to show that the sum is direct.
Recall that $\Lg=\bigoplus_{\al}(V_{\al}\oplus W_{\al})$.
Suppose that a subspace $V_{\theta}$ occurs in the decomposition of two ideals
$\Lg(\al)$ and $\Lg(\al')$. Then $\theta=\om^i\al=\om^j\al'$ for some $i,j\in \{0,1,2\}$.
Hence $\al^3=(\om^i\al)^3=\theta^3=(\om^j\al')^3=(\al')^3$, so that $\al\sim \al'$ and
$\Lg(\al)=\Lg(\al')$. Thus each $V_{\theta}$ occurs at most once in the above sum.
A similar argument shows that each $W_{\theta}$ appears at most once
in the sum. Hence the sum is direct.
\end{proof}

\begin{cor}\label{3.8}
Let $\Lg$ be a Lie algebra admitting a periodic derivation. Then $\Lg$ is hexagonally
graded.
\end{cor}

\begin{proof}
Let $D$ be a periodic derivation of $\Lg$. Then $D$ is semisimple and we may consider
the eigenspace decomposition $\Lg=\oplus_{\al}\Lg_{\al}$ of $\Lg$ with respect to $D$.
If $[\Lg_{\al},\Lg_{\be}]\neq 0$, then $\al,\be$ and $\al+\be$ are eigenvalues of $D$, hence
$m$-th roots of unity. By lemma $\ref{omega}$ we have $\be=\om \al$. Hence $\Lg$ is triangularly
graded and the claim follows from proposition $\ref{hextrig}$.
\end{proof}

\begin{cor}\label{3.9}
Let $\Lg$ be a Lie algebra admitting a nonsingular derivation whose inverse is
again a derivation. Then $\Lg$ is hexagonally graded.
\end{cor}

\begin{proof}
If there is such a derivation $D$, then we may assume that $D$ and $D^{-1}$ are semisimple,
by replacing them with the corresponding semisimple parts which are again nonsingular derivations. 
Consider again the eigenspace decomposition $\Lg=\oplus_{\al}\Lg_{\al}$ of $\Lg$ with respect to $D$.
If $[\Lg_{\al},\Lg_{\be}]\neq 0$, then $\al^{-1},\be^{-1}$ are eigenvalues of $D^{-1}$, and
$[\Lg_{\al},\Lg_{\be}]$ this is an eigenvector of $D^{-1}$ with eigenvalue
$\al^{-1}+\be^{-1}=(\al+\be)^{-1}$. The latter equation again implies that $\be=\om \al$. 
Hence $\Lg$ is triangularly graded and the claim follows from proposition $\ref{hextrig}$.
\end{proof}

Now we can summarize our characterizations as follows.

\begin{thm}\label{3.10}
For a complex Lie algebra the following statements are equivalent.
\begin{itemize}
\item[(1)] $\Lg$ admits a periodic derivation.
\item[(2)] $\Lg$ is hexagonally graded.
\item[(3)] $\Lg$ admits a nonsingular derivation whose inverse is
again a derivation. 
\end{itemize}
\end{thm}

\begin{proof}
First, $(1)$ implies $(2)$ by corollary $\ref{3.8}$. Then $(2)$ implies $(3)$ by lemma $\ref{3.4}$
and proposition $\ref{inverse}$. Finally $(3)$ implies $(1)$ by corollary $\ref{3.9}$ and 
lemma $\ref{3.4}$.
\end{proof}

Another consequence is the following characterization.

\begin{cor}
A Lie algebra admits a periodic derivation if and only if it admits a periodic
derivation of order six.
\end{cor}

\begin{proof}
We may assume that the Lie algebra $\Lg$ is nonabelian. Suppose that $\Lg$ admits a 
periodic derivation. Its order is a multiple of six. Then $\Lg$ is hexagonally graded, so that
lemma $\ref{3.4}$ yields a periodic derivation of order six. 
\end{proof}

\begin{cor}
If $\Lg$ admits a periodic derivation of order six, then it also admits
a periodic derivation of order $6k$ for all $k\ge 1$.
\end{cor}

\begin{proof}
If $D$ is a periodic derivation of order six, then $D=\diag(\al_1,\ldots , \al_n)$ with 
$\al_i$ being sixth roots of unity. By multiplying $D$ with a suitable sixth root of unity, we may 
assume that one of the $\al_i$ equals $1$, say $\al_j=1$. If we then multiply by a primitive 
$6k$-th root of unity, for each $k\ge 1$, then we obtain a diagonal derivation of order 
exactly $6k$, since the $j$-th entry has order exactly $6k$ in $\C^{\times}$.
\end{proof}

\section{Quotients of $N(2,g)$ by homogeneous ideals}

Every two-step nilpotent Lie algebra $\Lg$ is a quotient of $\Lf=N(2,g)$ by some ideal $I$
contained in $[\Lf,\Lf]$. We always assume that $\Lg$ is nonabelian, so that $g\ge 2$.
The dimension $r=\dim (I)\ge 0$ is an invariant of $\Lg$, the number of relations. We have
\[
\dim (\Lf)=g+\binom{g}{2},\; \dim (\Lg)=g+\binom{g}{2}-r.
\]
Denote by $\CX$ a minimal generating subset of $\Lf$, with $g$ elements.

\begin{defi}\label{partit} An ideal $J$ of $\Lf=N(2,g)$ is said to \emph{partition} 
$\Lf$ \emph{homogeneously}, 
if there exists a partition $\CX=X\cup Y\cup Z$ of $\CX$ into three subsets $X,Y,Z$ such that 
$$ J = J_X + J_Y + J_Z + J_{X,Y} + J_{X,Z} + J_{Y,Z}, $$ 
where $J_X = \langle [x,x'] | x,x' \in X \rangle$ and $J_{X,Y}$ is a linear subspace of 
$\langle [x,y] | x \in X, y \in Y\rangle$, and so on.
\end{defi}

As an example, consider $\Lf$ with $g=2m$ generators, and partition
\[
\CX=X\cup Y \cup Z=\{x_1,\ldots ,x_m\}\cup \{y_1,\ldots ,y_m\}\cup \emptyset .
\]
Define
\begin{align*}
J_X & =\langle [x_i,x_j] \mid 1\le i,j \le m\rangle ,\\
J_Y & =\langle [y_i,y_j] \mid 1\le i,j \le m\rangle ,\\
J_{X,Y} & = \langle [x_i,y_i]-[x_j,y_j]\mid 1\le i,j \le m\rangle,
\end{align*}
and $J_Z=J_{X,Z}=J_{Y,Z}=0$. Then  $J = J_X + J_Y + J_{X,Y}$ partitions $\Lf$ homogeneously,
and the quotient $\Lf/J$ is a two-step nilpotent Lie algebra with $2m$ generators and
$1$-dimensional commutator, which is obviously the Heisenberg Lie algebra $\Lh_m$ of 
dimension $2m+1$.

We obtain another characterization of Lie algebras admitting  a periodic derivation as follows.

\begin{prop}
A Lie algebra is hexagonally graded if and only if it is a quotient of $N(2,g)$ by a
homogeneously partitioning ideal.
\end{prop}

\begin{proof}
If $\Lg=N(2,g)/J$ is such a quotient, then it is easy to check that
\begin{align*}
\Lg & =\Lg_{\al}\oplus \Lg_{\be}\oplus \Lg_{\ga}\oplus \Lg_{\al+\be}\oplus \Lg_{\al+\ga}\oplus 
\Lg_{\be+\ga} \\
   & = \langle X \rangle \oplus \langle Y \rangle \oplus \langle Z \rangle \oplus
 \langle [X,Y] \rangle \oplus \langle [X,Z] \rangle \oplus  \langle [Y,Z] \rangle \mod J
\end{align*}
yields a well-defined hexagonal grading. \\
Conversely, assume that $\Lg$ is hexagonally graded. Let $X=\{x_1,\ldots ,x_r \}$ be
a basis for $\Lg(\al)$, $Y=\{y_1,\ldots ,x_s \}$ a basis for $\Lg(\be)$ and
$Z=\{z_1,\ldots ,z_t \}$ a basis for $\Lg(\ga)$. We may assume that
\[
[\Lg,\Lg]=\spa \{ [x_i,y_j],[x_i,z_j],[y_i,z_j]\},
\]
so that $\CX=X\cup Y\cup Z$ is a minimal generating set for $\Lg$ as a Lie algebra.
Denote by $\Lf$ the free-nilpotent Lie algebra of class $2$ on the generators
$X_1,\ldots ,X_r$, $Y_1,\ldots ,Y_s$, $Z_1,\ldots ,Z_t$. There exists a Lie algebra
epimorhism $\pi\colon \Lf\ra \Lg$ mapping $X_i$ to $x_i$, $Y_i$ to $y_i$ and
$Z_i$ to $z_i$. Thus $\Lg$ is isomorphic to $\Lf/J$ with $J=\ker (\pi)$. It suffices to show
that $J$ partitions $\Lf$ homogeneously. First observe that $J\subseteq [\Lf,\Lf]$. So
every $R \in J$ is of the form $$R = R_x + R_y + R_z + R_{xy} + R_{xz} + R_{yz},$$
where $R_x \in \spa([x_i,x_j])$, $R_{xy} \in \spa ([x_i,y_j])$, etc. For $X_i$ and $X_j$ in $X$, 
we obtain $\pi([X_i,X_j]) = [\pi(X_i),\pi(X_j)] = [x_i,x_j] = 0$ so that $[X_i,X_j] \in J$ for all 
$X_i,X_j$. Similarly, $[Y_i,Y_j], [Z_i,Z_j] \in J$. So we already know that $R_x,R_y,R_z \in J$. 
It now suffices to show that $R_{xy}, R_{xz}$ and $R_{yz}$ are also contained in $J$. 
Note that $$0 = \pi(R) = \pi(R_{xy}) + \pi(R_{xz}) + \pi(R_{yz}).$$ Since the terms in this 
sum belong to the linearly independent subspaces
\[
\spa([x_i,y_j]) = \Lg(\al + \be), \; \spa([x_i,z_j]) = \Lg(\al+\ga),\;
\spa([y_i,z_j]) = \Lg(\be+\ga)
\]
respectively, they must vanish. We conclude that $R_x,R_y,R_z,R_{xy}, R_{xz}$ and $R_{yz}$ are 
contained in $J$. 
\end{proof}

Theorem $\ref{3.10}$ implies the following result.

\begin{cor}
A Lie algebra admits a periodic derivation if and only if it is a quotient of $N(2,g)$ by a
homogeneously partitioning ideal.
\end{cor}

\begin{prop}
Let $\Lg=\Lf/I$ of dimension $n$ with $g$ generators and $\dim (I)=r$. 
Assume that $\Lg$ admits a periodic derivation. 
Then we have the following estimates:
\begin{align*}
g & \le n\le \frac{g^2}{3}+g, \\[0.2cm]
\frac{g(g-3)}{6} & \le r\le \frac{g(g-1)}{2}. 
\end{align*}
\end{prop}

\begin{proof}
By theorem $\ref{3.10}$, $\Lg$ is hexagonally graded, so that we may assume that $I$ partitions 
$\Lf$ homogeneously.
Let $X,Y$ and $Z$ as in the definition $\ref{partit}$. Then $g = |X| + |Y| + |Z|$ is a partition of $g$. 
Since $I$ contains the span of $[x,x'],[y,y']$ and $[z,z']$ for $x,x' \in X;y,y'\in Y;z,z' \in Z$, its 
dimension $r$ is at least $\binom{|X|}{2} + \binom{|Y|}{2} + \binom{|Z|}{2} \geq g (g - 3) / 6$. 
On the other hand we always have $r\le \dim ([\Lf,\Lf])=g(g-1)/2$. This shows the second estimate.
The first one follows form this by substituting $r = g + \binom{g}{2} - n$.
\end{proof}

We can use this result to classify Lie algebras $\Lg=\Lf/I$ with small $r$, admitting
a periodic derivation. The second estimate is equivalent to  
$\frac{\sqrt{8r+1}+1}{2} \le g\le \frac{\sqrt{24r+9}+3}{2}$. For a given small $r$ it restricts the possible 
values for $g$ very much. If $\Lg$ is nonabelian, then we have $\frac{\sqrt{8r+1}+1}{2} < g$.

\begin{prop}
Let $\Lg=N(2,g)/I$ be a two-step nilpotent Lie algebra with small $r=\dim (I)$, namely 
$0\le r\le 2$. Then $\Lg$ admits a periodic derivation if and only if it is isomorphic 
to a Lie algebra in the table below.
\vspace*{0.5cm}
\begin{center}
\begin{tabular}{c|c|c|c}
$r$ & $g$  & $\Lg$ & $\dim (\Lg)$ \\
\hline
$0$ & $2$ & $N(2,2)$ & $3$ \\
$0$ & $3$ & $N(2,3)$ & $6$ \\
$1$ & $3$ & $N(2,3)/\langle [x_2,x_3] \rangle$ & $5$ \\
$1$ & $4$ & $N(2,4)/\langle [x_1,x_2] \rangle$ & $9$ \\
$2$ & $3$ & $N(2,3)/\langle [x_1,x_2],[x_1,x_3] \rangle$ & $4$ \\
$2$ & $4$ & $N(2,4)/\langle [x_1,x_2],[x_3,x_4] \rangle$ & $8$ \\
    &     & $N(2,4)/\langle [x_2,x_4],[x_3,x_4] \rangle$ & $8$ \\
$2$ & $5$ & $N(2,5)/\langle [x_1,x_2],[x_3,x_4] \rangle$ & $13$ 
\end{tabular}
\end{center}
\end{prop}

\begin{proof}
All two-step nilpotent Lie algebras with $0\le r\le 2$ and $2\le g\le 5$ have been classified
in \cite{GAU}. An easy calculation then shows which of them admit a periodic derivation.
The result also follows by considering the possible homogeneously partitioning ideals and 
classifiying the resulting quotients.
\end{proof}

\begin{rem}
For $r=3$ the result becomes much more complicated. Then the above estimate gives $4\le g\le 6$.
The case $r=3$, $g=4$ is again classified in \cite{GAU}, Theorem $7.15$. 
There are six two-step nilpotent Lie algebras $N(2,4)/I_k$, $1\le k\le 6$
of dimension $7$. The only one {\it not} admitting a periodic derivation is  $N(2,4)/I_5$,
see example $\ref{2.13}$. This yields another interesting
fact. Admitting a periodic derivation is not a degeneration invariant. 

A Lie algebra $\Lg$ is said 
to {\it degenerate} to a Lie algebra $\Lh$, if $\Lh$ lies in the orbit closure of $\Lg$ under the action
of the general linear group acting by base changes. 
In fact, $N(2,4)/I_6$ degenerates to $N(2,4)/I_5$, but $N(2,4)/I_6$ admits a periodic derivation, 
whereas $N(2,4)/I_5$ does not.
\end{rem}

It is also possible to determine the Lie algebras with small commutator subalgebra, admitting
a periodic derivation.

\begin{prop}
Let $\Lg$ be a two-step nilpotent Lie algebra with $\dim ([\Lg,\Lg])\le 2$.
Then $\Lg$ admits a periodic derivation.
\end{prop}

\begin{proof}
It is well-known that a nilpotent Lie algebra with $1$-dimensional commutator subalgebra
is isomorphic to the direct sum of the Heisenberg Lie algebra $\Lh_m$ and some abelian
Lie algebra $\C^k$. Since both summands admit a periodic derivation, so does $\Lg$.
For the case that $\dim ([\Lg,\Lg])= 2$ we can refer to theorem $1$ in \cite{BDF}.
It says that $\Lg$ admits a hexagonal grading, and hence also a periodic derivation.
\end{proof}

\begin{rem}
As example $\ref{2.13}$ with $N(2,4)/I_5$ shows, a two-step nilpotent Lie algebra $\Lg$ with 
$3$-dimensional commutator subalgebra need not admit a periodic derivation.
\end{rem}

\section{Periodic prederivations}

Lie algebra prederivations are a generalization of Lie algebra derivations.
They have been studied in connection with Lie algebra degenerations, Lie triple systems
and bi-invariant semi-Riemannian metrics on Lie groups, see \cite{BU14} and the references
given therein. 

\begin{defi}
A linear map $P: \Lg \ra \Lg$ is called a {\it prederivation} of $\Lg$ if
$$P([x,[y,z]])=[P(x),[y,z]]+[x,[P(y),z]]+[x,[y,P(z)]]$$
for every $x,y,z\in \Lg$. It is called {\it periodic}, if $P^m=\id$ for some
$m\ge 1$.
\end{defi}

If $\Lg$ has a periodic prederivation $P$ with $m=1$, then $\Lg$ is at most two-step
nilpotent. Conversely, a Lie algebra of nilpotency class at most two admits
periodic prederivations of any any possible order $m\ge 1$.
Jacobsons result of \cite{JAC} generalizes to prederivations, and even to $k$-Leibniz
derivations, see \cite{MOE}. This implies the following result.

\begin{prop}
Let $\Lg$ be a Lie algebra admitting a periodic prederivation. Then $\Lg$ is nilpotent.
\end{prop}

Consider two first examples of nilpotency class $4$, namely the filiform nilpotent 
Lie algebras $\Lg_1$ and $\Lg_2$ of dimension $5$. 
The Lie brackets of $\Lg_1$ are given by $[x_1,x_i]=x_{i+1}$ for $i=2,3,4$. For $\Lg_2$ 
there is the additional bracket $[x_2,x_3]=x_5$.

\begin{ex}
The Lie algebra $\Lg_1$ admits periodic prederivations of any possible even 
order $m\ge 2$, whereas $\Lg_2$ admits no periodic prederivations.
\end{ex}

Indeed, $P=\diag(\al,-\al,-\al,\al,\al)$ with $\al^m=(-\al)^m=1$ is a periodic
prederivation of $\Lg_1$. On the other hand, a prederivation of $\Lg_2$ has the
eigenvalues $\al,\be,2\be-\al,2\al-\be,\al+2\be$. These cannot be all $m$-th roots of
unity. \\[0.2cm]
A Lie algebra admits periodic prederivations of {\it odd} order only in a trivial way.
To show this, we need another lemma on roots of unity.

\begin{lem}\label{5.4}
There exists $m$-th roots of unity $\al,\be,\ga$ such that $\al+\be+\ga$ is again
an $m$-th root of unity if and only if $m$ is even.
\end{lem}

\begin{proof}
If $m$ is even then we may take $\al=\be=-\ga=1$. Conversely, suppose that 
$\al^m=\be^m=\ga^m=\de^m=1$ with $\de=-(\al+\be+\ga)$ and $\al+\be+\ga+\de=0$.
We obtain a rhombus, so that two sides are vectors of opposite direction, say
$\be=-\al$. Then $1=\be^m=(-\al)^m=(-1)^m$ and hence $m$ is even.
\end{proof}

\begin{prop}
A Lie algebra $\Lg$ admits a periodic prederivation of odd order if and only if 
$\Lg$ it is nilpotent of class at most two.
\end{prop}

\begin{proof}
Assume that $[\Lg,[\Lg,\Lg]]\neq 0$ and $P^m=\id$ for some odd $m\ge 1$. Then $P$ is 
semisimple, and there exist eigenvectors $x,y,z$ with eigenvalues $\al,\be,\ga$
such that $[x,[y,z]]\neq 0$. In particular,  $[x,[y,z]]$ is an eigenvector with eigenvalue
$\al+\be+\ga$. Then $\al,\be,\ga$ and $\al+\be+\ga$ are $m$-th roots of unity.
This contradicts lemma $\ref{5.4}$ since $m$ is odd. It follows that  $[\Lg,[\Lg,\Lg]]=0$.
\end{proof}

\begin{rem}
Let $\Lg$ be a Lie algebra admitting a periodic prederivation $P$. Then it is not
difficult to show that the inverse $P^{-1}$ is again a prederivation of $\Lg$.
On the other hand, if $\Lg$ admits a nonsingular prederivation $P$, such that
$P^{-1}$ is again a prederivation, we do not know whether or not this implies the existence
of a periodic prederivation. Possibly theorem $\ref{3.10}$ can be extended to prederivations.
\end{rem}

We introduce the following class of Lie algebras.

\begin{defi}
A Lie algebra $\Lg$ is called an {\it pre-Engel-$m$} Lie algebra, if
the subspace $E_m(\Lg)=\s \{x\in \Lg\mid \ad(x)^m=0 \}$ has full dimension, i.e., 
if $\dim (E_m(\Lg))=\dim (\Lg)$.
\end{defi}

In other words, $\Lg$ is a pre-Engel-$m$ Lie algebra, if it admits an ad-nilpotent
basis of degree $m$. A special class if given by Engel-$m$ Lie algebras. These are
Lie algebras $\Lg$, where $\ad(x)^m=0$ holds for all $x\in \Lg$. They have received
a lot of attention in the literature, in particular concerning the possible nilpotency
class of such algebras. For $m\ge 5$ this is a difficult question. For $m=4$ the answer
is, that an Engel-$4$ Lie algebra has nilpotency class at most $7$, and there are such
examples. For details see \cite{TRA} and the references given therein.

\begin{rem}
Note that a simple Lie algebra may have an ad-nilpotent basis, e.g., $\Lg=\Ls\Ll_2(\C)$
with $e_1=\left(\begin{smallmatrix} 0 & 1 \\ 0 & 0 \end{smallmatrix}\right)$, 
$e_2=\left(\begin{smallmatrix} 0 & 0 \\ 1 & 0 \end{smallmatrix}\right)$ and
$e_3=\left(\begin{smallmatrix} \phantom{-}1 & \phantom{-}1 \\ -1 & -1 \end{smallmatrix}\right)$.
However, here we restrict ourself to nilpotent Lie algebras.
\end{rem}

We are in particular interested in nilpotent pre-Engel-$4$ Lie algebras, because
they are important for the existence of periodic prederivations. Clearly, any Lie algebra
of nilpotency class $c\le 4$ is a pre-Engel-$4$ Lie algebra. Thus it is more interesting
to consider pre-Engel-$4$ Lie algebras of nilpotency class $c\ge 5$. For the filiform
nilpotent case and the free-nilpotent case it is easy to obtain a classification.

\begin{prop}
A filiform Lie algebra $\Lg$ of nilpotency class $c\ge 5$ is a pre-Engel-$4$ Lie algebra
if and only if $\Lg$ is $6$-dimensional and $2$-step solvable.
\end{prop}

\begin{proof}
We have $\dim (\Lg)\ge 6$ because of $c\ge 5$. Let $\Lg$ be a filiform nilpotent
Lie algebra of dimension $n\ge 7$ and consider an adpated basis $e_1,\ldots ,e_n$, in the sense of 
Vergne. In particular we have $\ad (e_1)^{n-3}\neq 0$ and $\ad(e_1)^{n-2}=0$.
Let $x=\sum_{i=1}^n \la_ie_i$ and assume that $\ad (x)^4=0$. Then it is easy to see that
we always obtain $\la_1=0$. Here we need $n\ge 7$. This implies that 
$E_4(\Lg)\subseteq \s \{e_2,\ldots ,e_n\}$, so that $\Lg$ cannot be a pre-Engel-$4$ Lie 
algebra. Hence we are left with the case $\dim (\Lg)=6$ and $c=5$, where we have $5$ different 
algebras. Two of them are $3$-step solvable and three of them are $2$-step solvable.
The result follows by an easy computation.
\end{proof}

\begin{prop}
The free-nilpotent Lie algebra $N(c,g)$ is a pre-Engel-$4$ Lie algebra if and only if
$c\le 4$.
\end{prop}

\begin{proof}
We may assume that $g\ge 2$. Suppose that $c\ge 5$ and denote the generators by $x_1,\ldots,
x_g$. Let $e_1,\ldots ,e_n$ be {\it any} basis of $\Lf=N(c,g)$. It contains a subset which is a basis
for $\Lf/[\Lf,\Lf]$. We may assume that this subset is given by $\{ e_1,\ldots ,e_g\}$.
Define a Lie algebra automorphism $\al\colon \Lf\ra \Lf$ by $\al(x_i)=e_i$ for all $1\le i\le g$
(see \cite{GAU}). Therefore the identity $\ad (e_i)^t(e_j)=0$
for $i<j$ is equivalent to the identity $\ad (x_i)^t(x_j)=0$ for $i<j$, which holds if and only if 
$t\ge c$. It follows that $\ad (e_i)^4(e_j)\neq 0$ for all $1\le i<j \le g$. Hence 
$N(c,g)$ is not a pre-Engel-$4$ Lie algebra for $c\ge 5$.
\end{proof}

Pre-Engel-$4$ Lie algebras are related to the existence of periodic prederivations as follows.

\begin{prop}\label{5.10}
Let $\Lg$ be a nilpotent Lie algebra which is not a pre-Engel-$4$ Lie algebra.
Then $\Lg$ does not admit a periodic derivation.
\end{prop}

\begin{proof}
Assume that $\Lg$ admits a periodic prederivation $P$. Then $P$ is semisimple, and $\Lg$ 
admits a basis of eigenvectors $x_1,\ldots ,x_n$ with corresponding eigenvalues
$\al_1,\ldots ,\al_n$ of absolute value one. By assumption every basis $x_1,\ldots ,x_n$ 
contains an element $x_i$ with $\ad (x_i)^4\neq 0$. Hence $y=[x_i,[x_i,[x_i,[x_i,x_j]]]]$
is nonzero for two eigenvectors $x_i$ and $x_j$, and $y$ is again an eigenvector with
$P(y)=(4\al_i+\al_j)y$. Since $P$ is periodic, $\abs{4\al_i+\al_j}=1$. This is a contradiction
to $\abs{\al_i}=\abs{\al_j}=1$.
\end{proof}

\begin{cor}\label{5.11}
Let $\Lg$ be a filiform Lie algebra of nilpotency class $c\ge 5$. Then $\Lg$ does not 
admit a periodic prederivation.
\end{cor}

\begin{proof}
We have $\dim (\Lg)\ge 6$ because of $c\ge 5$. All such filiform Lie algebra, except
for three algebras of dimension $6$ are not pre-Engel-$4$ Lie algebras. Hence they
do not admit a periodic prederivation by proposition $\ref{5.10}$. For the remaining
three algebras the claim can be verified directly.
\end{proof}

\begin{cor}\label{free-n}
The free-nilpotent Lie algebra $N(c,g)$ with $c\ge 5$ does not admit a periodic prederivation.
\end{cor}

At this point there is the interesting question whether the above result can be generalized to
all nilpotent Lie algebras of class $c\ge 5$. In low dimensions the answer is positive. Indeed,
such algebras only exist in dimensions $n\ge 6$, and for $n=6$ they are filiform.
The first interesting dimension then is seven.

\begin{prop}
Let $\Lg$ be a nilpotent Lie algebra of dimension $7$ and nilpotency class $c\ge 5$.
Then $\Lg$ does not admit a periodic prederivation.
\end{prop}

\begin{proof}
If $c=6$, then $\Lg$ is filiform and the claim follows from corollary $\ref{5.11}$.
For $c=5$ there is a list of indecomposible algebras, consisting of $30$ single algebras 
and one family $\Lg(\la)$ of algebras (see for example \cite{MAG}). The claim is clear for 
those algebras which are not pre-Engel-$4$ Lie algebras. There remain $14$ pre-Engel-$4$ Lie 
algebras to be checked, where we have
computed all prederivations. The result is that none of these algebras admits a periodic
prederivation. For decomposible algebras we have $\Lg=\Lf\oplus \C$, where $\Lf$ is filiform
of dimension $6$. Again a direct computation gives the result.
\end{proof}

In general, however, the result of corollary $\ref{5.11}$ cannot be generalized.
We present a counter example in dimension $8$: take the free-nilpotent Lie algebra
$N(2,5)$ with $2$ generators $x_1,x_2$ and nilpotency class $c=5$, and consider the quotient
$\Lg$ given by the following Lie brackets,

\begin{align*}
x_3 & = [x_1,x_2], \\
x_4 & = [x_1, [x_1,x_2]]=[x_1,x_3], \\
x_5 & = [x_2, [x_1,x_2]]=[x_2,x_3], \\
x_6 & = [x_2,[x_1, [x_1,x_2]]]=[x_2,x_4], \\
x_6 & = [x_1,[x_2, [x_1,x_2]]]=[x_1,x_5], \\
x_7 & = [x_2,[x_2, [x_1,x_2]]]=[x_2,x_5], \\
x_8 & = [x_1,[x_2,[x_2, [x_1,x_2]]]]=[x_1,x_7], \\
x_8 & = [x_2,[x_1,[x_2, [x_1,x_2]]]]=[x_2,x_6].\\
\end{align*}

Note that this Lie algebra is a pre-Engel-$4$ Lie algebra, with $\ad (x_i)^4=0$ for
$1\le i\le 8$. However, it is not an Engel-$4$ Lie algebra because of $\ad (x_1+x_2)^4\neq 0$.

\begin{ex}\label{5.13}
The above $5$-step nilpotent Lie algebra admits periodic prederivations of every
even order $m\ge 2$.
\end{ex}

Indeed, a direct computation shows that for all $\al,\be,\ga \in \C$,
\[
P=\diag (\al,\be,\ga,2\al+\be,\al+2\be,\al+\be+\ga,2\be+\ga,2\al+3\be)
\]
is a prederivation of $\Lg$. For $\be=-\al$, $\ga=\al$ we obtain
\[
P=\diag (\al,-\al,\al,\al,-\al,\al,-\al,-\al).
\]
With $\al$ being a primitive $m$-th root of unity satisfying $(-\al)^m=1$, this prederivation is 
periodic of order $m$. \\[0.2cm]
There is another class of Lie algebras not admitting periodic prederivations.

\begin{defi}
Let $\Lg$ be a Lie algebra. We say that $\Lg$ satisfies {\it property $F$}, if
every basis contains a triple $(x_1,x_2,x_3)$ of basis elements such that for 
all $i\neq j$ in $\{1,2,3\}$ holds: $[x_i,[x_i,x_j]]\neq 0$ or $[x_j,[x_j,x_i]]\neq 0$.
\end{defi}

For a given Lie algebra, it seems difficult to check this for every basis.
Therefore it is more convenient to rewrite the definition as follows.
A Lie algebra $\Lg$ does {\it not} satisfy property $F$ if it has a basis $x_1,\ldots ,x_n$
such that for all triples $(x_i,x_j,x_k)$ there exists a subset $\{x_{\ell},x_m\}\subseteq
\{ x_i,x_j,x_k\}$ such that 
\[
[x_{\ell},[x_{\ell},x_m]]=[x_m,[x_m,x_{\ell}]]=0.
\]

\begin{ex}
The Lie algebra of example $\ref{5.13}$ does not have property $F$.
\end{ex}

To see this, we need to find a basis $y_1,\ldots ,y_8$ such that for each of
the possible $\binom{8}{3}=56$ triples $(y_i,y_j,y_k)$ there is a pair $(y_{\ell},y_m)$ such that
$\ad(y_{\ell})^2(y_m)=\ad(y_{m})^2(y_{\ell})=0$. It turns out that we may choose the original
basis $x_1,\ldots ,x_8$ of $\Lg$ together with the pairs  $(x_{\ell},x_m)$ for
$(\ell,m)=(1,2),(1,3),(2,3),(1,4),(2,4),(3,4),$ $(5,6),(5,7),(5,8),(6,7),(6,8),(7,8)$.

\begin{prop}\label{5.16}
The free-nilpotent Lie algebra $N(c,g)$ has property $F$ if and only if either $c\ge 3,g\ge 3$ 
or $c\ge 4,g=2$.
\end{prop}

\begin{proof}
Suppose first that $c\ge 3$ and $g\ge 3$. Denote the generators of $\Lf=N(c,g)$ by $x_1,\ldots,
x_g$. Let $e_1,\ldots ,e_n$ be any basis of $\Lf$. It contains a subset, say  $\{ e_1,\ldots ,e_g\}$,
which is a basis for $\Lf/[\Lf,\Lf]$. We obtain a Lie algebra automorphism $\al\colon \Lf\ra \Lf$ 
by setting $\al(x_i)=e_i$ for all $1\le i\le g$. We may choose three different vectors from
$\{ e_1,\ldots ,e_g\}$, say $e_1,e_2,e_3$. Since $[x_i,[x_i,x_j]]\neq 0$ for all $i\neq j$ in
$\{1,2,3\}$, we obtain  $[e_i,[e_i,e_j]]\neq 0$ for all $i\neq j$ in $\{1,2,3\}$. This shows that
$\Lf$ has property $F$. \\
Now suppose that $c\ge 4$ and $g=2$. Denote the generators by $x_1,x_2$ and let $e_1,\ldots ,e_nx$ 
be again any basis of $\Lf$, such that $e_1$ and $e_2$ are a basis of $\Lf/[\Lf,\Lf]$. 
We may choose $e_3$ such that $e_1,e_2,e_3$ is a basis of $\Lf/[\Lf,[\Lf,\Lf]]$. As before, define 
an automorphism by $\al(e_i)=x_i$ for $i=1,2$. Then we may write 
\[
\al(e_3)=\la_1 x_1+\la_2x_2+\la_3[x_1,x_2]+v
\]
for some $\la_i\in \C$ with $\la_3\neq 0$ and $v\in [\Lf,[\Lf,\Lf]]$. A short computation shows that
the three terms $[e_1,[e_1,e_2]],[e_1,[e_1,e_3]], [e_2,[e_2,e_3]]$ are nonzero, by applying $\al$ and using
$[x_i,[x_i,x_j]]\neq 0$ for all $i\neq j$. Hence $\Lf$ has property $F$.
\end{proof}

The relation to periodic prederivations is as follows.

\begin{prop}\label{5.17}
Let $\Lg$ be a Lie algebra having property $F$. Then $\Lg$ does not admit a periodic prederivation.
\end{prop}

\begin{proof}
Suppose that $\Lg$ admits a periodic prederivation $P$. Then $P$ is semisimple, and there
is a basis of eigenvectors $e_1,\ldots ,e_n$ with $P(e_i)=\al_i e_i$ and $\abs{\al_i}=1$ for all 
$i=1,\ldots ,n$.
Since $\Lg$ has property $F$, there exists a triple $(e_1,e_2,e_3)$ such that for all $i\neq j$ in
$\{ 1,2,3\}$ either $[e_i,[e_i,e_j]]$ or $[e_j,[e_j,e_i]]$ (or both) are nonzero eigenvalues, i.e.,
either $\abs{2\al_i+\al_j}=1$ or $\abs{2\al_j+\al_i}=1$. In other words, for all $i\neq j$ we have 
$\al_i=-\al_j$. It follows $\al_1=\al_2=\al_3=0$, which is a contradiction.
\end{proof}

\begin{cor}
The free-nilpotent Lie algebra $N(c,g)$ admits a periodic prederivation if and only if
$c\le 2$, or if $c=3,g=2$.
\end{cor}

\begin{proof}
If not $c\le 2$ or $c=3,g=2$, then $N(c,g)$ does not admit a periodic prederivation
by propositions $\ref{5.16}$ and  $\ref{5.17}$. Conversely, if $c\le 2$, then any
endomorphism is a prederivation. Hence there exists a periodic prederivation.
If $c=3,g=2$, then we have the Lie algebra $N(3,2)$ with basis $x_1,\ldots ,x_5$ and
brackets $[x_1,x_2]=x_3$,  $[x_1,x_3]=x_4$, $[x_2,x_3]=x_5$. A periodic prederivation is
given by $P=\diag (\al,-\al,1,\al,-\al)$. 
\end{proof}


\end{document}